\documentclass[11pt]{amsart}

\usepackage{latexsym,amsmath,amssymb,amsfonts,amscd,graphics}
\usepackage[mathscr]{eucal}
\usepackage{mathrsfs}
\usepackage{tikz-cd}
\usepackage{hyperref}
\usepackage{booktabs}
\usepackage{tree-dvips}
\usepackage{array}
\usepackage{colortbl}
\usepackage{hhline}
\usepackage{multirow}
\usepackage[width=\textwidth]{caption}
\hypersetup{colorlinks=true,
	linkcolor=blue,
	filecolor=blue,      
	urlcolor=cyan,
	citecolor=magenta}

\numberwithin{equation}{section}
\theoremstyle{plain}
\newtheorem{theorem}[equation]{Theorem}

\newtheorem{lemma}[equation]{Lemma}
\newtheorem{proposition}[equation]{Proposition}
\newtheorem{corollary}[equation]{Corollary}

\theoremstyle{remark}
\newtheorem{remark}[equation]{Remark}

\theoremstyle{definition}
\newtheorem{definition}[equation]{Definition}

\newtheorem{example}[equation]{Example}

\newcommand{\bP}{\mathbb{P}}

\newcommand{\bQ}{\mathbb{Q}}
\newcommand{\bZ}{\mathbb{Z}}

\newcommand{\bC}{\mathbb{C}}

\newcommand{\bL}{\mathbb{L}}

\newcommand{\calC}{\mathcal{C}}

\newcommand{\calO}{\mathcal{O}}

\newcommand{\calJ}{\mathcal{J}}

\newcommand{\Aut}{\mathrm{Aut}}
\newcommand{\Bir}{\mathrm{Bir}}

\newcommand{\Pic}{\mathrm{Pic}}

\newcommand{\git}{/\kern-0.2em/}
\newcommand{\Lneg}{L_{-}}

\newcommand{\OG}{\Lambda}
\newcommand{\leech}{\mathbb{L}}
\newcommand{\markman}{\mathcal{W}^{pex}_{OG10}}

\DeclareMathOperator{\rank}{\mathrm{rank}}

 \title[]{Classification of Symplectic Birational Involutions of manifolds of $OG10$ type} 

\author[L. Marquand]{Lisa Marquand}
\address{
	Courant Institute,
	251 Mercer Street,
	New York, NY 10012, USA
}
\email{lisa.marquand@nyu.edu}

\author[S. Muller]{Stevell Muller}
\address{Mathematik und Informatik
	Geb\"aude E.2.4 - Raum: 221
	Universit\"at des Saarlandes
	66123 Saarbrücken (Germany)}
\email{muller@math.uni-sb.de}


\begin{document}
	\bibliographystyle{alpha}
	
	\begin{abstract}
		We give a complete classification of symplectic birational involutions of manifolds of $OG10$ type. We approach this classification with three techniques --- via involutions of the Leech lattice, via involutions of cubic fourfolds and finally lattice enumeration via a modified Kneser's neighbour algorithm. The classification consists of three involutions with an explicit geometric realisation via cubic fourfolds, and three exceptional involutions which cannot be obtained by any known construction. 
	\end{abstract}
 \maketitle
	\section{Introduction}
The classification of symplectic automorphisms of irreducible holomorphic symplectic manifolds has been widely studied. In his celebrated paper \cite{MR958597}, Mukai classified symplectic automorphisms of a $K3$ surface. This was further streamlined by Kond\=o \cite{MR1620514}, who related automorphisms of $K3$ surfaces to automorphisms of the Niemeier lattices. In recent years, there has been intense work on classifying symplectic automorphisms of higher dimensional irreducible holomorphic symplectic manifolds. Using a similar approach to Kond\=o, Mongardi obtained a classification of prime order symplectic automorphisms of manifolds of $K3^{[n]}$ type \cite{MR3102529, MR3473636}. Similar results were obtained by Huybrechts \cite{huybrechtsderived}. A classification of symplectic automorphisms of manifolds of $OG6$ type was obtained in \cite{grossi2020finite}. In contrast, manifolds of $OG10$ type admit no regular symplectic automorphisms, as shown recently in \cite{OG10rigid}. For a $K3$ surface, the group of automorphisms and the group of birational transformations coincide; in higher dimensions this is no longer the case. One can consider instead the group of symplectic birational transformations of an irreducible holomorphic symplectic manifold.

In this paper, we investigate symplectic birational involutions of manifolds deformation equivalent to O'Grady's ten dimensional exceptional example ($OG10$ type) \cite{MR1703077}. Our interest is motivated by our desire to study the fixed loci; each irreducible component inherits an induced holomorphic symplectic structure. In the case of involutions, the associated moduli space of $OG10$ manifolds with an involution of a given type is a type IV period domain. In this way we obtain variations of Hodge structures of $K3$ type. This does not occur for higher order cyclic groups (see \cite{LPZ,MR4190414}).

Let $X$ be a manifold of $OG10$ type, and let $f\in \Bir(X)$ be a symplectic birational involution. We obtain an induced involution on the second cohomology $\OG:=H^2(X,\bZ)$, determining two lattices $\OG_+,\, \OG_-$, the invariant and the coinvariant lattices respectively. Vice versa, specifying such lattices (subject to certain lattice theoretic conditions) determines a symplectic birational transformation of some manifold of $OG10$ type via the Global Torelli Theorem (Theorem \ref{thm: no markman so sympl}). Our main theorem is a classification of symplectic birational involutions of manifolds of $OG10$ type. We classify them by distinguishing their induced action on the discriminant group $A_\Lambda:=\Lambda^*/\Lambda$. In the column labelled \textit{Geo.} of Table \ref{main table}, we indicate whether we can geometrically realise the involution via known constructions of $OG10$ type manifolds.

\begin{theorem}\label{main thm}
Let $X$ be a manifold of $OG10$ type, $f\in\Bir(X)$ a symplectic birational involution. Then the induced action $\iota$ on $\Lambda:=H^2(X,\bZ)$ is determined by one of the cases (1)-(6) below.

\begin{center}
			\vspace{-1.5mm}
   \small\centering\setlength{\tabcolsep}{4pt}
\renewcommand\arraystretch{1.5}
			\begin{table}[h!]
				\begin{tabular}{llllll}
					\toprule
					  & $\OG_+$ & $\OG_-$ & $\rank(\OG_-)$ & Action on $A_\Lambda$ & Geo.\\ 
					\toprule
					(1) & $U^3\oplus D_4^3(-1)$ & $E_6(-2)$ & $6$ & nontrivial & yes\\
                    (2) & $U^3\oplus A_2(-1)\oplus E_8(-2)$ & $E_8(-2)$ & $8$ & trivial & yes\\
                    (3) & $U^2\oplus A_1\oplus A_1(-1)\oplus E_8(-2)$ & $M(-1)$ & $10$ & nontrivial & yes\\
                    (4) & $U^2(2)\oplus A_1\oplus A_1(-1)\oplus E_6(-2)$ & $D_{12}^+(-2)$ & $12$ & trivial & no\\
                    (5) & $\langle 2\rangle^3\oplus\langle -2\rangle^9$ & $G_{12}$ & $12$ & nontrivial & no\\
                    (6) & $\langle 2\rangle^3\oplus\langle -2\rangle^5$ & $G_{16}$ & $16$ & nontrivial & no\\
					\bottomrule
				\end{tabular}
        \caption{Classification of birational involutions of manifolds of $OG10$ type}\label{main table}
			\end{table}
			\vspace*{-\baselineskip}
		\end{center}
  Here $M$ is the unique index two overlattice of $D_9(2)\oplus\langle 24\rangle$ (see \cite{marquand2022cubic} for an explicit description).
  
   Moreover, each involution of $\Lambda$ listed above is unique up to conjugacy in $O(\OG)$, and there exists a manifold of $OG10$ type with a birational involution inducing such an isometry.
  
\end{theorem}

The proof of Theorem \ref{main thm} is divided into three cases (Theorems \ref{trivial on AL} , \ref{partial class}, and \ref{non-trivial non split}). 
First, in Theorem \ref{trivial on AL}, we classify symplectic birational involutions acting trivially on the discriminant group by using the same techniques of \cite{huybrechtsderived,MR3473636}; we relate these involutions to involutions of the Leech lattice. 
This recovers the Nikulin type involution with coinvariant lattice $E_8(-2)$, and more interestingly, we obtain an involution with coinvariant lattice $D_{12}^+(-2)$ that cannot be realised in the case of $K3$ surfaces. Thus we obtain the involutions labelled (2) and (4) in Table \ref{main table}. 

Secondly, in Theorem \ref{partial class}, we classify symplectic birational involutions acting nontrivially on the discriminant group whose coinvariant lattice has $\rank(\OG_-) <12$. In this case, we reduce the classification to the classification of involutions of cubic fourfolds. More precisely, under the rank assumption such an involution induces an involution on a smaller lattice, which can be identified with the primitive middle cohomology of some cubic fourfold via the Torelli theorem of Voisin \cite{voisintorelli}. The coinvariant lattice $\OG_-$ can then be identified with the invariant lattice of an antisymplectic involution of the cubic fourfold, which have been classified (see \cite{LPZ,marquand2022cubic}). The only possibilities are the involutions labelled (1) and (3) in Table \ref{main table}.

Thirdly, in Theorem \ref{non-trivial non split}, we classify symplectic birational involutions that act nontrivially on the discriminant group, with $\rank(\OG_-)\geq 12$. 
With this hypothesis, it is fairly easy to identify the possible lattices $\OG_+$; however, the corresponding lattices $\OG_-$ are not unique in their genera. 
There are 12 possible genera for $\OG_-$; we first enumerate the possible lattices in each genus, and then exclude all but two possibilities, obtaining the involutions labelled (5) and (6) in Table \ref{main table}.
The results of this enumeration are displayed in Table \ref{table enum}. 
We accomplish this enumeration using a modification of Kneser's neighbour method \cite{KneserMartin2002QF}; this is computer aided (see Appendix \ref{appendix:A} for more details). 
Our enumeration results can be found in the database \cite{marquand_lisa_2023_7528193}.

In all three cases, we obtain involutions $\iota$ of the lattice $\Lambda$ acting as in cases (1)-(6). We conclude that there exists a manifold $X$ of type $OG10$ with a birational involution $f\in \Bir(X)$ inducing $\iota$ using a corollary of the Torelli Theorem (Theorem \ref{thm: no markman so sympl}). However, we are able to geometrically realise the involutions in cases (1)-(3) through known constructions, using the construction of \cite{LSV}.
More precisely,  using the observation of Sacc\`a \cite[\S 3.1]{sac2021birational} (see also \cite{LPZ}), an involution of a cubic fourfold induces a birational transformation of the corresponding compactified intermediate Jacobian $X$, a manifold of $OG10$ type by \cite{LSV,sac2021birational}, giving a geometric realisation of the involutions in cases (1)-(3). 

\subsection*{Outline} We recall the relevant definitions and previous results in Section \ref{prelims}. In Section \ref{sec: trivial act} we consider symplectic birational involutions acting trivially on the discriminant group of $\OG$ (Theorem \ref{trivial on AL}), and obtain the involutions (2) and (4). In Section \ref{sec: non-trivial <12}, we obtain a classification of symplectic birational involutions that act nontrivially on the discriminant group of $\Lambda$ and satisfy $\rank(\OG_-)<12$ (Theorem \ref{partial class}), and obtain the involutions (1) and (3). In Section \ref{sec:geo} we geometrically realise the involutions (1)-(3) via the construction of \cite{LSV}. In Section \ref{sec: non-trivial > 12}, we obtain a classification of symplectic birational involutions acting nontrivially on the discriminant group of $\OG$ but have $\rank(\OG_-)\geq 12$ (Theorem \ref{non-trivial non split}), and obtain the involutions (5) and (6). This concludes the proof of Theorem \ref{main thm}. Details of the enumeration algorithm required for the proof of Theorem \ref{non-trivial non split} can be found in Appendix \ref{appendix:A}.

\subsection*{Notations}
Throughout, all lattices are assumed to be integral and even unless stated otherwise. Let $L$ be a lattice and $\iota\in O(L)$ an isometry of order 2. We denote by $L_+$ and $L_{-}$ the invariant and coinvariant lattices defined by $\iota$ respectively. For any lattice $L$, we denote by $L(n)$ the lattice with the bilinear form scaled by $n\in\bZ_{\neq0}$. The discriminant group of a lattice is denoted by $A_L=L^*/L$. We denote by $\textnormal{div}_L(v)=\max\{\lambda\in \bZ\mid v\cdot L\subset \lambda \bZ\}$ the divisibility of an element $v\in L$. We assume all $ADE$ lattices are positive definite. We denote by $\leech$ the Leech lattice, which is the unique even negative definite unimodular lattice without roots. A 2-elementary lattice $L$ is determined by its signature, and the invariants $\delta_L$ and $l(A_L)$. Here $\delta_L\in\{0,1\}$ with $\delta_L=0$ if and only if $q_{A_L}$ takes values in $\bZ/2\bZ$, and $l(A_L)$ is the minimum number of generators of $A_L$.

\subsection*{Acknowledgments} The first author would like to thank her advisor, Radu Laza, for suggesting to her this problem as well as for the invaluable advice and suggestions. She would also like to thank Giulia Sacc\`a for useful discussions, as well as Giovanni Mongardi, Annalisa Grossi, Luca Giovenzana, Claudio Onorati and Davide Veniani for their comments and suggestions on the previous version. The second author would like to thank his advisor, Simon Brandhorst, for the precious help provided regarding the work mentioned in the appendix. Finally, we thank the referees for their feedback and suggestions that improved the quality of this manuscript.

The first author was partially supported by NSF Grant DMS-2101640 (PI Laza). The second author was supported by the Deutsche Forschungsgemeinschaft (DFG, German Research Foundation) [Gefördert durch die Deutsche Forschungsgemeinschaft (DFG) – Projektnummer 286237555 – TRR 195].

\section{Preliminaries}\label{prelims}
In this section we collect preliminary results. In Section \ref{subsec: OG10} we recall the definition of manifolds of $OG10$ type. 
In Section \ref{subsec: bir}, we recall a criteria for  when an isometry of the lattice $\OG$ is induced by a birational involution, based on the Torelli theorem.
In Section \ref{subsec: bir inv} we give further necessary conditions on the existence of induced isometry depending on the action on the discriminant group of $\OG$. 
We recall the relevant results on cubic fourfolds and the construction of \cite{LSV} and \cite{sac2021birational} in Section \ref{cubic prelims}.

\subsection{Manifolds of \texorpdfstring{$OG10$}{OG10} type}\label{subsec: OG10}
An \textbf{irreducible holomorphic symplectic} manifold is a simply connected, compact, K\"ahler manifold $X$ such that $H^0(X, \Omega^2_X)$ is generated by a nondegenerate holomorphic 2-form $\sigma$. 
Let $X$ be an irreducible holomorphic symplectic manifold that is deformation equivalent to O'Grady's 10-dimensional exceptional example \cite{MR1703077}. 
Then we say $X$ is a manifold of $OG10$ type.

It follows from the definition of irreducible holomorphic symplectic manifolds that $H^2(X,\bZ)$ is a torsion-free $\bZ$-module; equipped with the Beauville--Bogomolov--Fujiki form $q_X$ it becomes a lattice.
By \cite{MR2349768}, for $X$ a manifold of $OG10$ type there is an isometry $(H^2(X,\bZ),q_X)\cong \OG,$ where $$\OG:=U^3\oplus E_8(-1)^2\oplus A_2(-1).$$ 
A \textit{marking} is a choice of an isometry $\eta:H^2(X,\bZ)\rightarrow \OG$.
The purpose of this paper is to classify symplectic birational involutions of manifolds $X$ of $OG10$ type, in terms of their action on $H^2(X,\bZ)\cong \OG$. 
We always assume that a manifold of $OG10$ type has a fixed marking $\eta$ throughout.

\subsection{Torelli Theorem}\label{subsec: bir}
We denote by $ \Bir(X)$ the group of birational transformations of $X$ respectively. For an irreducible holomorphic symplectic manifold $X$, a birational transformation $f\in \Bir(X)$ is well defined in codimension one. We thus obtain an isometry $f^*:H^2(X,\bZ)\rightarrow H^2(X,\bZ)$.
\begin{definition}\label{defn: symplectic}
	We say a birational transformation $f\in \Bir(X)$ is \textbf{symplectic} if the induced action $f^*:H^2(X,\bC)\rightarrow H^2(X,\bC)$ acts trivially on $\sigma$. 
\end{definition}

Assume now that $X$ is a manifold of $OG10$ type, and consider the associated orthogonal representation
$$\eta_*:\Bir(X)\rightarrow O(\OG);\,\,\,\,\, f\mapsto \eta\circ (f^*)^{-1}\circ\eta^{-1}.$$
Note that $\eta_\ast$ is injective by  \cite[Theorem 3.1]{MR3592467}.

\begin{definition}
	An isometry $g\in O(\OG)$ is \textbf{induced} by a birational transformation if there exists a manifold $X$ of $OG10$ type and $f\in \Bir(X)$ such that $\eta_*(f)=g$.
\end{definition}

A birational transformation of $X$ preserves the birational K\"ahler cone $\mathcal{BK}(X)$. This in turn imposes restrictions on which involutions of the lattice $\OG$ are induced by birational involutions of such a manifold $X$. The structure of the birational K\"ahler cone for a manifold of $OG10$ type is now fully understood \cite{mongardi2020birational}.

In particular, the walls of $\overline{\mathcal{BK}(X)}$ are defined by the hyperplanes $D^\perp\subset \calC(X)$, where $D$ is a \textbf{stably prime exceptional divisor} \cite[\S 5]{markman}, and $\calC(X)$ denotes the connected component of the positive cone of $X$ containing a K\"ahler class. 
  We define the following set of vectors:
 \[\markman:=\{v\in\OG : v^2=-2\}\cup \{v\in \OG : v^2=-6,\; \textnormal{div}_{\OG}(v)=3\}.\]
  \begin{proposition}[{{{\cite[Proposition 3.1]{mongardi2020birational}}}}]
      Let $(X, \eta)$ be a marked manifold of $OG10$ type. Then $D\in\Pic(X)$ effective is stably prime exceptional if and only if $\eta(D)\in \markman.$
 \end{proposition}
  
  It follows that $\mathcal{BK}(X)$ is contained in an \textbf{exceptional chamber}; that is, a component of \[\mathcal{C}(X)\setminus \bigcup_{v\in \markman} v^\perp\](see \cite[Theorem 3.2]{mongardi2020birational}).
  Using this description, we can rephrase the Global Torelli theorem (due to Huybrechts, Markman and Verbitsky) in a way that is more suited for the study of symplectic birational transformations of $X$. This provides us with criteria for when an isometry $g\in O(\OG)$ is induced. 

	\begin{theorem}[{{{\cite[Theorems 2.15 and 2.17]{grossi2020finite}}}}]\label{thm: no markman so sympl}
	An involution $g\in O(\OG)$ is induced by a symplectic birational transformation if and only if $\OG_{-}$ is negative definite and 
	$$\OG_{-}\cap \markman=\varnothing.$$
\end{theorem}
\begin{proof}
    Although stated for manifolds of $OG6$ type, the proof of \cite[Theorems 2.15 and 2.17]{grossi2020finite} apply verbatim in the case of $OG10$ manifolds.
\end{proof}

\begin{remark}
    All birational symplectic transformations of a manifold of $OG10$ type are nonregular --- indeed, by \cite{OG10rigid}, a manifold of $OG10$ type has no nontrivial symplectic automorphisms of finite order. 
\end{remark}

\subsection{Discriminant action}\label{subsec: bir inv}
In order to classify symplectic birational involutions of manifolds of $OG10$ type, we consider two cases corresponding to the induced action of $\iota\in O(\OG)$ on the discriminant group $$A_{\OG}:=\OG^*/\OG\cong \bZ/3\bZ.$$ It follows that an involution acts by $\iota|_{A_{\OG}}=\pm \textnormal{id}_{A_{\OG}}$.
\begin{remark}
	Note that $A_{\OG}=A_{A_2(-1)}$; let $A_2(-1)$ be generated by $\alpha_1, \alpha_2$ where $\alpha_i^2=-2,$ and $\alpha_1\cdot \alpha_2=1$. Then $A_{A_2(-1)}\cong\bZ/3\bZ$ is generated by \[\gamma:=\left[\frac{2\alpha_1+\alpha_2}{3}\right]\] and $q_{A_2(-1)}(\gamma)= -\frac{2}{3}+2\bZ$. 
\end{remark}
\begin{proposition}\label{properties of Lneg}
	Let $X$ be a manifold of $OG10$ type, let $f\in \Bir(X)$ be a symplectic birational involution, and let $\iota=\eta_*(f)\in O(\OG)$ the induced isometry of $f$. Then $\OG_{-}$ is a negative definite lattice of rank $r\leq 21$, with $\OG_{-}\cap \markman=\varnothing$, and the following hold:
	\begin{enumerate}
		\item If $\iota$ acts trivially on $A_{\OG}$, then $\OG_{-}$ is a 2-elementary, negative definite lattice whose genus is determined by the invariants $(r,l(A_{\OG_{-}}),\delta_{\OG_{-}})$. 
		\item If $\iota$ acts by $-\textnormal{id}$ on $A_{\OG}$, then $\OG_+$ is a 2-elementary lattice with signature $(3,21-r)$.
	\end{enumerate}
\end{proposition}
\begin{proof}
	The negative definiteness and the claim that $\OG_{-}\cap \markman=\varnothing$ follows from Theorem \ref{thm: no markman so sympl}. Claim $(1)$ follows by \cite[Lemma 2.8]{grossi2020finite}; for claim $(2)$ consider $\iota':=-\iota$; it follows that $\OG_+$ is 2-elementary.
\end{proof}
\subsection{Cubic Fourfolds}\label{cubic prelims}
Cubic fourfolds lead to irreducible holomorphic symplectic manifolds through various constructions, and one can study birational transformations induced by automorphisms of the cubic. This was first studied by Camere \cite{MR2967237} in her work on symplectic involutions of the Fano variety of lines, an irreducible holomorphic symplectic variety \cite{MR818549}. 

We use this idea for manifolds of $OG10$ type. Let $V\subset \bP^5$ be a smooth cubic fourfold, and let $\pi_U:J_U\rightarrow U\subset (\bP^5)^{\vee} $ be the Donagi--Markman fibration; i.e the family of intermediate Jacobians of the smooth hyperplane sections of $V$. The total space $J_U$ has a holomorphic symplectic form, by \cite{MR1397273}. The main result of \cite{LSV} is the construction, for a general $V$, of a smooth projective irreducible holomorphic symplectic compactification $\calJ_V$ of $J_U$, with a Lagrangian fibration $\pi:\calJ_V\rightarrow (\bP^5)^{\vee}$ extending $\pi_U$. It was shown that $\calJ_V$ is an irreducible holomorphic symplectic manifold of $OG10$ type. This result was extended to every smooth cubic fourfolds \cite{sac2021birational}, to obtain the following theorem.

\begin{theorem}[{{{\cite{LSV,sac2021birational}}}}]
	Let $V\subset \bP^5$ be a smooth cubic fourfold, and let $\pi_U:J_U\rightarrow U\subset (\bP^5)^{\vee}$ be the Donagi--Markman fibration. Then there exists a smooth projective irreducible symplectic compactification $\calJ_V$ of $J_U$ of $OG10$ type with a morphism $\pi:\calJ_V\rightarrow (\bP^5)^{\vee}$ extending $\pi_U$.
\end{theorem}
We note that the same result holds for the irreducible holomorphic symplectic compactification $\calJ_V^T$ of the nontrivial $J_U$-torsor $J^T_U\rightarrow U$ of \cite{MR3832411}. 

Recall that the primitive cohomology $H^4(V,\bZ)_{prim}$ admits a Hodge structure of $K3$-type (up to a Tate twist). 
In particular, as a lattice $$H^4(V,\bZ)_{prim}\cong U^2\oplus E_8^2\oplus A_2.$$ 
We define the lattice of primitive algebraic cycles $$A(V)_{prim}:=H^{2,2}(X)\cap H^4(X,\bZ)_{prim}.$$ 
The main result of \cite{marquand2022cubic} (see also \cite{LPZ}) is a classification of involutions of a smooth cubic fourfold in terms of the sublattice $A(V)_{prim}.$ 
\begin{theorem}[{{{\cite[Theorem 1.1]{marquand2022cubic}}}}]\label{marquandthm}
	Let $V\subset\bP^5$ be a  general cubic fourfold with $\phi_i$ an involution of $V$ fixing a linear subspace of $\bP^5$ of codimension $i$. Then either:
	\begin{enumerate}
		\item $i=1$, $\phi_1$ is antisymplectic and $A(V)_{prim}\cong E_6(2);$
		\item $i=2$, $\phi_2$ is symplectic and $A(V)_{prim}\cong E_8(2);$
		\item $i=3$, $\phi_3$ is antisymplectic and $A(V)_{prim}\cong M.$
	\end{enumerate}
	Here $M$ is the unique rank 10 even lattice obtained as an index 2 overlattice of $D_9(2)\oplus\langle 24\rangle$.
\end{theorem}

 \section{Trivial action on the Discriminant group}\label{sec: trivial act}

Throughout, we let $X$ be a manifold of $OG10$ type, and $\OG\cong H^2(X,\bZ)$. We prove the following:
\begin{theorem}\label{trivial on AL}
	Let $X$ be a manifold of $OG10$ type, and $f\in\Bir(X)$ be a symplectic birational involution. Suppose that $\iota:=\eta_*(f)$ acts trivially on the discriminant group $A_{\OG}$. Then one of the following holds:
	\begin{enumerate}
		\item $\OG_{-}\cong E_8(-2)$ and $\OG_+\cong  U^3\oplus E_8(-2)\oplus A_2(-1)$; or 
		\item $\OG_{-}\cong D_{12}^+(-2)$ and $\OG_+\cong U^2(2)\oplus A_1\oplus A_1(-1)\oplus E_6(-2)  $.
	\end{enumerate}

  Moreover, each involution of $\Lambda$ listed above is unique up to conjugacy in $O(\OG)$, and there exists a manifold of $OG10$ type with a birational involution inducing such an isometry.
\end{theorem}
The strategy to prove Theorem \ref{trivial on AL} is as follows: we first consider arithmetic involutions $\iota\in O(\OG)$ such that $\iota$ acts trivially on $A_{\OG}$, and $\OG_{-}$ is negative definite.
In Section \ref{sec: Leech} we use techniques of Kond\={o} and Mongardi to embed the coinvariant lattice $\OG_{-}$ into the Leech lattice $\leech$. 
We extend the involution $\iota$ to one of the Leech lattice $\leech$, and use the classification of involutions \cite{MR1078499} to obtain three candidates. 
In Section \ref{sec: proof of 3.1} we prove Theorem \ref{trivial on AL} by case-by-case analysis. More precisely, we show that only $E_8(-2)$ and $D_{12}^+(-2)$ are realised as coinvariant lattices $\OG_{-}$ for an involution of $\OG$. 
We then show $\OG_{-}\cap \markman=\varnothing$, and conclude by Theorem \ref{thm: no markman so sympl} that such an involution $\iota$ is induced by a geometric symplectic birational involution $f\in\Bir(X)$ of a manifold $X$ of $OG10$ type. This completes the proof of Theorem \ref{trivial on AL}, and cases (2) and (4) of Theorem \ref{main thm}.

\subsection{The Leech Lattice}\label{sec: Leech} We reduce the classification of involutions $\iota\in O(\OG)$ acting trivially on $A_{\OG}$ to a classification of involutions of the Leech lattice $\leech$. 
\begin{proposition}\label{embed into leech}
	Let $\iota\in O(\OG)$ be an involution acting trivially on $A_{\OG}$ such that $\OG_-$ is negative definite and such that $\OG_{-}\cap \markman=\varnothing$. Then there exists a primitive embedding of $\OG_{-}$ into the Leech lattice $\leech$. Further, there exists an involution of the Leech lattice $\leech$ such that $\leech_-\cong\OG_{-}.$
\end{proposition}
\begin{proof}
    By assumption, we know that $\Lambda_-$ is negative definite, it contains no vectors $v$ such that $v^2 = -2$ and it has rank at most 21 (since $\OG$ has real signature $(3,21)$). Since $\OG_-$ is the coinvariant lattice of $\iota$, which acts trivially on $A_\OG$, we also know that $\OG_-$ is 2-elementary, the isometry $\iota|_{\OG_-} = -\textnormal{id}_{\OG_-}$
    fixes no nontrivial vectors in $\Lambda_-$ and it acts trivially on $A_{\OG_-}$. Moreover 
    \[\gcd(|A_\OG|,\, |A_{\OG_-}|) = 1\] 
    so a direct application of \cite[Proposition 1.15.1]{nikulin} gives us that
    \[\rank\OG_-+l(A_{\OG_-}) \leq \rank\OG_-+l(A_{\OG_+}) \leq \rank \OG_-+\rank \OG_+ = \rank\OG = 24.\]
    Hence under these conditions, similar results as in \cite[Corollary 4.19]{papergroupog10} hold; meaning that there exists a primitive embedding $j\colon\OG_-\hookrightarrow\bL$, and $-\textnormal{id}_{\OG_-}$ extends to an involution of $\bL$ so that $\bL_- = j(\OG_-)$ (up to replacing $O^\#(\Lambda_{\OG_-})$ by $\langle-\textnormal{id}_{\OG_-}\rangle$ in the proof of \cite[Corollary 4.19]{papergroupog10}).
\end{proof}

The nontrivial involutions $\iota\in O(\leech)$ are classified: 

\begin{proposition}[{{{\cite[Table 1]{MR1078499}}}}]\label{cases}
	There exist exactly three conjugacy classes of nontrivial involutions of the Leech lattice $\leech$. They are classified by specifying their invariant/coinvariant lattices:\begin{enumerate}
		\item $\leech_-\cong E_8(-2)$ and $\leech_+\cong BW_{16}$;
		\item $\leech_-\cong D_{12}^+(-2)$ and $\leech_+\cong D_{12}^+(-2)$; or
        \item $\leech_-\cong BW_{16}$ and $\leech_+\cong E_8(-2)$.
	\end{enumerate}
\end{proposition}
We have three possible candidates for $\OG_{-}$ as above. It remains to be seen whether there exists an involution $\iota\in O(\OG)$ whose coinvariant lattice is the given candidate. We show such an involution exists provided there exists a primitive embedding of each candidate into $\OG$. 
\begin{lemma}\label{inv}
	Let $M$ be a 2-elementary lattice with a primitive embedding $M\hookrightarrow L$ into a lattice $L$ and let
    % such that 
    $N:=(M)^\perp_L$. Then there exists an involution $\iota\in O(L)$ such that the coinvariant lattice $\Lneg=M$ and the invariant lattice $L_+=N$.
\end{lemma}
\begin{proof}
	By assumption we have that 
	$$M\oplus N\hookrightarrow L\hookrightarrow L\otimes \bQ\cong (M\oplus N)\otimes \bQ.$$
	
	We can define $\iota_\bQ:L_\bQ\rightarrow L_\bQ$ by $\iota(x)=-x$ for $x\in M$, and $\iota(x)=x$ for $x\in N$. We want to show that $\iota_\bQ$ is defined over $L$. By assumption $L/(M\oplus N)\cong (\bZ/2\bZ)^a$, and thus for all $x\in L$, we have that $2x\in M\oplus N$.
	Let $x\in L$. By above, we can write $x=\frac{x_-+x_+}{2},$ with $x_-\in M, x_+\in N$. Thus $\iota_\bQ(x)= x \mod M\oplus N$,
	and $[\iota_\bQ(x)]=[x]$ in $L/(M\oplus N)$; thus $\iota_\bQ(x) \in L$. 
\end{proof}
It remains to be seen whether a primitive embedding exists in each of the cases in Proposition \ref{cases}. We require the following lemma:

\begin{lemma}\label{existence of prim embedding}
	Let $M$ be a 2-elementary, negative definite lattice of rank $r$ with invariants $(r,a,\delta)$, where $a=l(A_{M})$. Then there exists a primitive embedding $M\hookrightarrow \OG$ if and only if there exists a lattice $N$ of signature $(3, 21-r)$ satisfying the following properties:
	\begin{enumerate}
		\item $A_N\cong(\bZ/2\bZ)^a\oplus \bZ/3\bZ$;
		\item $q_N|_{(\bZ/2\bZ)^a}\cong-q_{M}$;
		\item $q_N|_{\bZ/3\bZ}\cong q_{\OG}$.
	\end{enumerate}
	We say $\delta:=\delta_N=0$ or $1$ if and only if $\delta_{M}=0$ or $1$. In this case, $N\cong (M)_\Lambda^\perp$
\end{lemma}
\begin{proof}
	This follows immediately from \cite[Prop 1.15.1]{nikulin}.
\end{proof}

\begin{theorem}\label{thm: prim embedding}
    Let $M \in \{E_8(-2),\, D_{12}^+(-2),\, BW_{16}\}$. There exists an involution of $\OG$ with $\OG_-\cong M$ if and only if $M = E_8(-2)$ or $D_{12}^+(-2)$. Moreover, such involution is unique up to conjugacy in $O(\OG)$.
\end{theorem}

\begin{proof}
    We show the existence of a primitive embedding into $\OG$ in the case of $M = E_8(-2),\, D_{12}^+(-2)$, and prove the nonexistence of such an embedding in the case of $M = BW_{16}$.

    First, consider the involution of $\OG$ defined by interchanging the two copies of $E_8(-1)$, and identity elsewhere. Then $\OG_{-}\cong E_8(-2)$ and thus $E_8(-2)$ embeds primitively into $\OG$ (see \cite[\S 5]{MR728142} for more details).

    Next, we show that $D_{12}^+(-2)$ embeds primitively into $\OG$ by applying Lemma \ref{existence of prim embedding}. Recall that the lattice $D_{12}^+(-2)$ is an even, 2-elementary lattice with signature $(0,12),$ $ a=12$ and $\delta=1$. We need to check whether or not the genus $g$ of even lattices determined by the signature pair $(3,9)$ and $q_{A_2(-1)}\oplus (-q_{D_{12}^+(-2)})$ is empty. We use the function \textbf{is\_genus} on Hecke \cite{nemo}, and obtain that $g$ is not empty and in particular, it contains the lattice
	$$N := U^2(2)\oplus A_1\oplus A_1(-1)\oplus E_6(-2) .$$
	The latter can be shown by computing the \emph{genus symbol} of a gram matrix associated to $N$ and comparing it to the symbol of $g$ (the genus symbol, in the convention of Conway \& Sloane, uniquely determines a genus, see \cite[Chapter 15, \S 7]{splg} for more details). This is implemented for instance on Hecke \cite[\textbf{genus}]{nemo}.

  Finally, we show that the Barnes--Wall lattice $BW_{16}$ does not embed primitively into $\OG$. Recall that $BW_{16}$ is an even 2-elementary lattice of signature $(0,16)$, $a=8$ and $\delta=0$.	By Lemma \ref{existence of prim embedding}, a primitive embedding $BW_{16}\hookrightarrow \OG$ exists if and only if there exists an even lattice $N$ of signature $(3,5)$ and discriminant group $A_N\cong(\bZ/2\bZ)^8\oplus \bZ/3\bZ$ such that $q_N|_{(\bZ/2\bZ)^8}$ takes values in $\bZ/2\bZ$ and $q_N|_{\bZ/3\bZ}=q_{A_2(-1)}$.
    Similarly to the $D_{12}^+(-2)$ case, one uses the function \textbf{is\_genus} to determine that the genus of even lattices with invariants $(3,5)$ and $q_N = q_{A_2(-1)}\oplus (-q_{BW_{16}})$ is empty and so we do not have a primitive embedding of $BW_{16}$ into $\OG$.

    To conclude, we note that Lemma \ref{existence of prim embedding} tells us that the genus of the complement of $M =E_8(-2),\, D_{12}^+(-2)$ in $\OG$ is uniquely determined. Moreover, according to \cite[Theorem 1.14.2]{nikulin}, such a genus consists of a unique isometry class, and for any representative $N$ of such isometry class, we have that $O(N)\to O(A_N)$ is surjective. Hence, by \cite[Corollary 1.5.2]{nikulin}, any two involutions of $\OG$ with $\OG_+ \cong N$ and $\OG_-\cong M$ are conjugate in $O(\OG)$.
 
\end{proof}

\subsection{Proof of Theorem \ref{trivial on AL}}\label{sec: proof of 3.1}

We have exhibited involutions $\iota\in O(\OG)$ acting trivially on $A_{\OG}$ with $\OG_{-}$ isomorphic to either $ E_8(-2)$ or $D_{12}^+(-2)$. In order to conclude that both involutions are induced by symplectic birational involutions of a manifold $X$ of $OG10$ type we must show that neither contain vectors $v$ with $v^2=-2,$ or $v^2=-6$ and $\textnormal{div}_\OG(v)=3$. 

\begin{lemma}\label{no long root}
	Let $\iota\in O(\OG)$ be an involution and suppose that $\OG_{-}$ contains a vector $v$ with $v^2=-6$ and $\textnormal{div}_{\OG}(v)=3$. Then $A_{\OG_{-}}$ contains an element of order 3.
\end{lemma}
\begin{proof}
	Let $v\in \OG_{-}$ be such that $v^2=-6$ and $\textnormal{div}_{\OG}(v)=3$. Then $3$ divides the divisibility of $v$ in $\OG_-$. We can write $\textnormal{div}_{\OG_{-}}(v)=3k$ for some positive integer $k$. Then $[v^*]=\left[\frac{v}{3k}\right]$ defines a nonzero element of $A_{\OG_{-}}$; in particular, $kv^*$ is a nontrivial element of order 3.
\end{proof}

\begin{proof}[Proof of Theorem \ref{trivial on AL}]
	The discriminant group of both $E_8(-2)$ and $D_{12}^+(-2)$ contains no elements of order three; by Lemma \ref{no long root} primitive vectors $v\in \OG_{-}$ with $v^2=-6$ and $\textnormal{div}_{\OG}(v)=3$ cannot exist. The maximal norm of both lattices is $-4$, and so they do not contain any roots. Thus in both cases $\OG_{-}$ is a negative definite lattice with $\OG_{-}\cap \markman=\varnothing$; we have also shown these are the only possible negative definite coinvariant lattices for an involution $\iota\in O(\OG)$ acting trivially on the discriminant group $A_\OG$. By Theorem \ref{thm: no markman so sympl} each such involutions is induced by a symplectic birational involution of a manifold of $OG10$ type. The classification of the corresponding invariant lattices follow from the proof of Theorem \ref{thm: prim embedding}.
\end{proof}

\begin{remark}
	In fact, the lattice $D_{12}^+(-2)$ embeds primitively into the Mukai lattice of a $K3$ surface, and so one also obtains a symplectic birational involution of a manifold of $K3^{[2]}$ type with this coinvariant lattice (see \cite{huybrechtsderived}, or \cite[\S 7.5]{MR4450620} for a different description).
\end{remark}

    \section{Nontrivial action on discriminant, rank < 12}\label{sec: non-trivial <12}
It remains to be seen whether there exist birational involutions of manifolds of $OG10$ type that act nontrivially on the discriminant group. In this section, we prove that this is indeed possible, and classify such involutions whose coinvariant lattice has small rank. More precisely, we prove the following:

\begin{theorem}\label{partial class}
	Let $X$ be a manifold of $OG10$ type, and $f\in \Bir(X)$ a symplectic birational involution, such that the induced action $\iota:=\eta_*(f)$ is nontrivial on the discriminant group $A_\OG$. Assume further that $$\rank\OG_-<12.$$ Then one of the following holds:
	\begin{enumerate}
		\item $\OG_-\cong E_6(-2)$ and $\OG_+\cong U^3\oplus D_4^3(-1)$; or
		\item $\OG_-\cong M(-1)$ and $\OG_+\cong U^2 \oplus A_1\oplus A_1(-1) \oplus E_8(-2) $,
	\end{enumerate}
where $M$ is the unique rank 10 lattice obtained as an index 2 overlattice of $D_9(2)\oplus \langle 24\rangle$.

 Moreover, each involution of $\Lambda$ listed above is unique up to conjugacy in $O(\OG)$, and there exists a manifold of $OG10$ type with a birational involution inducing such an isometry.
\end{theorem}

We prove Theorem \ref{partial class} as follows: in Section \ref{sec: split U} we use Nikulin's classification of 2-elementary lattices to develop a criterion for the invariant lattice $\OG_+$ to split a $U$ summand, i.e. $\OG_+=U\oplus \Gamma$ for some lattice $\Gamma$. This is always satisfied under the assumption that $\rank(\OG_-)<12$. The latter allows us to classify such involutions in Section \ref{sec:proof of partial class} by utilising Theorem \ref{marquandthm}, on the classification of involutions of a cubic fourfold \cite{LPZ},\cite{laza2019automorphisms},\cite{marquand2022cubic}. In particular, in the case of nontrivial action on $A_\OG,$ the lattice $\OG_-$ occurs as an invariant lattice for an antisymplectic involution on a cubic fourfold, leading to the possibilities above. Finally, we show that Theorem \ref{thm: no markman so sympl} applies, completing the proof of Theorem \ref{partial class}, and cases (1) and (3) of Theorem \ref{main thm}.

\subsection{Splitting a \texorpdfstring{$U$}{U} summand}\label{sec: split U}

Assume that $f\in\Bir(X)$ is a symplectic birational involution of a manifold $X$ of $OG10$ type, such that $\iota:=\eta_*(f)\in O(\OG)$ is an involution that acts by $-\textnormal{id}$ on $A_{\OG}$. 
By Proposition \ref{properties of Lneg}, $\OG_{-}$ is negative definite of rank $1\leq r\leq 21$, and $\OG_{+}$ is a 2-elementary lattice of signature $(3, 21-r)$. 
In order to prove Theorem \ref{partial class}, we first use Nikulin's classification of $2$-elementary lattices. We state the result only in the case that we require here.

\begin{lemma}[{{{\cite[Theorems 4.3.1]{nikulin_elementary},\cite[Theorem 1.5.2]{MR728992}}}}]\label{lem: 2elem}

    Let $N$ be a 2-elementary lattice of signature $(3,21-r)$, $A_N\cong (\bZ/2\bZ)^a$ and invariant $\delta:= \delta_N$. Then $N$ is unique in its genus, and the following hold:

    \begin{enumerate}
	\item $a\leq \min\{24-r, r\}$,
	\item $a\equiv r \mod 2$,
	\item $r\equiv 2\mod 4 $ if $\delta=0$,
	\item $\delta=0, r\equiv 2\mod 8$  if $a=0$,
	\item $r\equiv 3 \mod 8$ if $a=1$,
	\item $\delta=0$ if $a=2$, $r\equiv 6 \mod 8$,
	\item $r\equiv 2\mod 8$ if $\delta=0, a=24-r$.
\end{enumerate}
\end{lemma}

Applying the above lemma, we establish a numerical criteria for $\OG_+$ to split a $U$ summand.

\begin{lemma}\label{split U}
	Let $r=\rank \OG_{-}$, $a,\,\delta$ as above. Then $\OG_{+}$ splits of a $U$ summand if and only if:
	\begin{enumerate}
		\item $r\leq 20$, and $a\leq 22-r$;
		\item If $a=22-r$ and $\delta=0$, then $r\equiv 2 \mod 8$.
	\end{enumerate}
\end{lemma}
\begin{proof}
	Assume $\OG_+$ splits of a $U$ summand, i.e $\OG_+\cong N\oplus U$. Applying Nikulin's classification of 2-elementary lattices to the lattice $N$ with invariants $((2, 20-r), a, \delta)$, we see the above conditions are necessary for the existence of such a lattice $N$. Conversely, assume the conditions in the theorem hold. Then again by the classification, there exists a 2-elementary lattice $N$ with invariants $((2, 20-r), a, \delta)$. Then $N\oplus U$ has the same invariants as $\OG_+$, and thus are in the same genus. Since $\OG_+$ is indefinite, it is unique and the claim holds.
\end{proof}

\subsection{Proof of Theorem \ref{partial class}}\label{sec:proof of partial class}
The following lemma provides us with a simple way to determine when there exists a vector $v\in \OG_-$ with $v^2=-6$ and $\textnormal{div}_\OG(v)=3,$ which is necessary for the proof of Theorem \ref{partial class}.

\begin{lemma}\label{lem: characterisation long roots}
    Let $L$ be a lattice such that $A_L\cong \bZ/3\bZ$ and let $\iota\in O(L)$ be an involution acting nontrivially on $A_L$. Then for a vector $v\in L_-$ we have
    \[\textnormal{div}_L(v) = 3 \iff \textnormal{div}_{L_-}(v) \in 3\bZ.\]
\end{lemma}

\begin{proof}
    Let us denote $M := L_+$ and $N:= L_-$. One implication is already clear from Lemma \ref{no long root}. 
    Now let $v\in N$ primitive be such that $\textnormal{div}_N(v) \in 3\bZ$ and let $w\in L$. Since $\iota$ has order 2, we know that $2L\subseteq M\oplus N$. Thus there exist $w_M\in M$ and $w_N\in N$ such that $w = \frac{w_M+w_N}{2}$. Now, since $M$ and $N$ are in orthogonal direct sum in $L$, $(v, w_M) = 0$ and moreover, $(v, w_N)\in 3\bZ$. But now, $v\in N\subseteq L$ with $L$ integral, so we have that 
    \[\left(v, \frac{w_M+w_N}{2}\right) = \frac{1}{2}(v, w_N)\]
    is an integer. Hence, $(v, w_N)\in 2\mathbb{Z}\cap 3\bZ = 6\bZ$, and we deduce that $(v, w)\in 3\bZ$. Hence, since we chose $w\in L$ arbitrary, we conclude that $\textnormal{div}_L(v) = 3$.
\end{proof}

We prove Theorem \ref{partial class} by first proving that under the rank assumption, $\OG_+$ splits a $U$ summand, and then utilising the classification of involutions of a cubic fourfold in Theorem \ref{marquandthm}. Finally, we  apply Lemma \ref{lem: characterisation long roots} to conclude such an involution is induced geometrically.

\begin{proof}[Proof of Theorem \ref{partial class}]
    Let $f\in \Bir(X)$ be as in the assumptions of the statement, and denote by $\iota:=\eta_*(f)\in O(\OG)$ the induced involution on $\OG$. We claim the conditions of Lemma \ref{split U} are satisfied --- indeed, since $r:=\rank(\OG_-)<12,$ and by assumption $a\leq r\leq 22-r$ we only need to exclude the case $a=r=11$ and $\delta=0$. For the 2-elementary lattice $\Lambda_+$ with invariants $((3, 21-r),22-r, 0)$ to exist, by Lemma \ref{lem: 2elem} (3) we see that $r\equiv 2\mod 4,$ a contradiction. It follows that conditions of Lemma \ref{split U} are always satisfied, and that there exists a 2-elementary lattice $\Gamma$ such that $\OG_+\cong \Gamma\oplus U.$ 

    Let $U_1:=U$ be such that $\OG_+=\Gamma\oplus U_1\hookrightarrow \OG$. 
	Denote by $L=(U_1)^\perp_{\OG}(-1);$ then $L$ is an even, indefinite lattice with signature $(20,2)$ and discriminant group $A_{L}\cong \bZ/3\bZ\cong A_{A_2}$. 
    By \cite[Cor. 1.13.3]{nikulin}, $L$ is unique up to isometry; thus we see that $$L\cong U^2\oplus E_8^2\oplus A_2.$$ Since $\iota$ acts as the identity on $\Gamma\oplus U_1$, $\iota$ restricts to an isometry of $L$ with 
	$$L_+\cong \Gamma(-1) \text{ and } L_-\cong\OG_-(-1).$$ Note that $\OG_{-}$ is negative definite of rank $r\leq 20$, and $\Gamma$ has signature $(2,20-r)$. We can choose a polarized Hodge structure $H$ of weight 4 on the lattice $L$, of type $(0,1,20,1,0)$, such that $$H^{2,2}\cap L=\OG_-(-1).$$ Notice this implies $H^{3,1}\subset L_+$. 
    By assumption, $\OG_{-}\cap\markman=\varnothing$; in particular, according to Lemma \ref{lem: characterisation long roots}, we know that $\OG_-$ contains no vectors $v$ such that $v^2=-2$ or $v^2=-6$ and $\textnormal{div}_{\OG_-}(v)=3$. 
    This condition assures that the Hodge structure $H$ defines a point in the image of the period map for smooth cubic fourfolds by \cite[Theorem 1.1]{laz10}.
    It follows that there exists a smooth cubic fourfold $V$ and a marking $\gamma\colon H^4(V,\bZ)_{prim}\to L$ such that $\gamma(A(V)_{prim}) = \OG_-(-1)$. For simplicity, in what follows, we identify $H^4(V, \bZ)_{prim}$ with $L$ via the marking $\gamma$.
    
	Let $\eta_V\in H^4(V,\bZ)$ be the square of the hyperplane class. We wish to extend $\iota|_{L}$ to an isometry of $H^4(V,\bZ)$ fixing $\eta_V$. 
	We have that $L\oplus \langle \eta_V\rangle \subset H^4(V,\bZ)$ is an index 3 overlattice; in order to extend with the identity on $\langle\eta_V\rangle$, $\iota|_{L}$ must act trivially on $A_{L}\cong \bZ/3\bZ\cong A_{\langle \eta_V\rangle}$ \cite[Corollary 1.5.2]{nikulin}. 
    By assumption, $\iota|_{L}$ acts by $-\textnormal{id}$ on $A_{L}$. Set $\sigma:=-\iota|_L$; notice now that for the action of $\sigma$ on $H^4(V,\bZ)_{prim}$ we have:
	$$(H^4(V,\bZ)_{prim})_-=L_+\cong \Gamma(-1), \text{ and } (H^4(V,\bZ)_{prim})_+=L_-\cong \OG_-(-1).$$
	Now $\sigma\oplus \textnormal{id}_{\langle \eta_V\rangle}$ defines an isometry of $H^4(V,\bZ)$ fixing $\eta_V$; let us denote this by $\sigma_V$. 
 Thus $\sigma_V\in \Aut_{HS}(V,\eta_V)$, and by the strong version of the Global Torelli theorem for cubic fourfolds \cite[Prop 1.3]{ZHENG_2019}, there exists a unique automorphism $\phi\in \Aut(V)$ such that $\sigma_V=\phi^*$. 
 Notice that $\sigma_V$ acts nontrivially on $H^{3,1}$; the involution $\phi$ is antisymplectic for the cubic fourfold $V$, and by Theorem \ref{marquandthm}:
	$$\OG_-\cong(H^4(V,\bZ)_{prim}(-1))_+=\begin{cases}
		E_6(-2),\\
		M(-1).
	\end{cases}$$ 

 Finally, we prove existence of an $OG10$ manifold $X$ admitting such an involution. In both cases above, $\OG_-$ is negative definite.  
 Since $\OG_-(-1)$ appears as an invariant lattice for an involution of a smooth cubic fourfold (Theorem \ref{marquandthm}), it follows that $\OG_-\cap \markman=\varnothing$ (Lemma \ref{lem: characterisation long roots}). 
 Hence by Theorem \ref{thm: no markman so sympl}, the involution is induced geometrically by a symplectic birational transformation $f\in \Bir(X)$ for some manifold $X$ of $OG10$ type. Further we see that such an involution $\iota$ necessarily acts by $-\textnormal{id}$ on the discriminant group $A_{\OG}$; if $\iota$ acted trivially, then $\OG_-$ would be 2-elementary, a contradiction by Proposition \ref{properties of Lneg}. The determination of $\OG_+$ in both cases follows.

 Let us conclude by remarking that, as in the proof of Theorem \ref{thm: prim embedding}, for both involutions of the statement of Theorem \ref{partial class} the lattice $\OG_+$ is unique in its genus and $O(\OG_+)\to O(A_{\OG_+})$ is surjective. Hence, again according to \cite[Corollary 1.5.2]{nikulin}, any involution of $\OG$ with $(\OG_+, \OG_-)$ as in the statement of Theorem \ref{partial class} is unique up to conjugacy in $O(\OG)$.
\end{proof}

\begin{remark}\label{remaining}
    It follows that if $f\in \Bir(X)$ is a symplectic birational involution with $\OG_+=U\oplus \Gamma$ for some lattice $\Gamma$, then $\rank(\OG_-)<12.$ Indeed, one applies the same argument as in Theorem \ref{partial class} to obtain an involution of a cubic fourfold (see also Theorem \ref{reduce to cubic}). Thus to conclude the proof of Theorem \ref{main thm}, it suffices to classify involutions with $\rank(\OG_-)\geq 12$, nontrivial action on $A_\OG$ and the assumption that $\OG_+$ does not split a $U$ summand. We complete this classification in Section \ref{sec: non-trivial > 12}.
\end{remark}

\section{Geometrical realisations via cubic fourfold}\label{sec:geo}

 In this section, we provide a geometrical realisation of the involutions (1)-(3) of Theorem \ref{main thm}, Table \ref{main table}.  All of these examples are obtained via involutions of some cubic fourfolds --- we show that such involutions produce induced symplectic birational involutions of the associated compactified intermediate Jacobian.

\begin{theorem}\label{reduce to cubic}
	Let $X$ be an irreducible holomorphic symplectic manifold of OG10 type. Let $f\in\Bir(X)$ be a symplectic birational involution of $X$, and suppose that $H^2(X,\bZ)_+\cong \Gamma\oplus U$ for some lattice $\Gamma$.
	Then there exists a smooth cubic fourfold $V$ with an involution $\phi$ whose action is determined by $f$. In particular, one of the following holds: $$H^2(X,\bZ)_-\cong \begin{cases}E_6(-2),\\
		E_8(-2),\\
		M(-1).
	\end{cases}$$
	
	Conversely, an involution $\phi$ of a smooth cubic fourfold $V$ induces a symplectic birational involution $f$ on the compactified associated Intermediate Jacobian $\calJ_V$, that leaves a copy of $U$ invariant and whose action is determined by $\phi$.
\end{theorem}

\begin{proof}
    Denote by $\iota:=\eta_*(f)\in O(\Lambda)$ the induced involution on $\Lambda.$ The existence of such a smooth cubic fourfold in the case of nontrivial action on the discriminant group is contained in the proof of Theorem \ref{partial class}. Thus we can assume $\iota$ acts trivially on the discriminant group $A_\OG.$ The argument follows identically to that of Theorem \ref{partial class} to obtain a smooth cubic fourfold $V$ with an involution $\iota$ of $H^4(V,\bZ)_{prim}$.
    We extend it to an involution of $H^4(V,\bZ)$ with the identity on $\langle \eta_V\rangle,$ where $\eta_V$ is the square of the hyperplane class. 
    Again, the strong version of the Global Torelli theorem for cubic fourfolds \cite[Prop 1.3]{ZHENG_2019} implies that there exists a unique $\phi\in \Aut(V)$ such that $\iota\oplus \textnormal{id}_{\langle \eta_V\rangle}=\phi^*$, with $\phi$ symplectic. By Theorem \ref{marquandthm}: $$(H^4(V,\bZ)_{prim})_-\cong E_8(2)$$ (see also \cite[Theorem 1.2 (1)]{laza2019automorphisms}). Thus necessarily $\OG_-\cong E_8(-2).$

	Conversely, suppose we have an involution $\phi\in\Aut(V)$ of a smooth cubic fourfold $V\subset \bP^5$; let $\sigma$ be the induced involution on $H^4(V,\bZ)_{prim}$. By (\cite{LSV}, \cite[Theorem 1.6]{sac2021birational}), we can associate to $V$ an irreducible holomorphic symplectic manifold $\calJ_V$ of OG10 type, with a Lagrangian fibration $\pi: \calJ_V\rightarrow \bP^5$ that compactifies the intermediate Jacobian fibration of $V$. Note that the compactification $\calJ_V$ is not unique; the cubic fourfold $V$ is a special cubic fourfold containing either a plane or a cubic scroll \cite{marquand2022cubic}, and so may have many birational compactifications, as discussed in \cite{sac2021birational}.
	Let $\Theta$ denote the relative theta-divisor of $\calJ_V$; then the sublattice $\langle \Theta, \pi^*\calO(1)\rangle$ is isomorphic to the hyperbolic lattice $U$ \cite[Lemma 3.5, communicated by K.~Hulek, R.~Laza]{sac2021birational}.
	
	To obtain an involution of $\calJ_V$, we follow \cite[Sect. 3.1]{sac2021birational}. The automorphism $\phi\in\Aut(V)$ acts on the universal family of hyperplane sections of $V$, and thus on the Donagi--Markman fibration $\mathcal{J}_U\rightarrow U,$ where $U\subset (\bP^5)^{*}$ parametrises smooth hyperplane sections of $V$ \cite[Section 3.1]{sac2021birational}. We thus obtain in this way a birational transformation $f:\calJ_V\dashrightarrow \calJ_V$, that leaves the sublattice $\langle \Theta, \pi^*\calO(1)\rangle\cong U$ invariant.
	If $\phi\in \Aut(V)$ is symplectic (i.e acts trivially on $H^{3,1}(V)$), then the induced birational involution $f\in \Bir(\calJ_V)$ is symplectic, by \cite[Lemma 3.2]{sac2021birational}. If not, there exists a regular antisymplectic involution $\tau\in \Aut(\calJ_V)$ given geometrically by sending $x\mapsto -x$ on the smooth fibers of $\calJ_V\rightarrow \bP^5$. Further this involution $\tau$ commutes with the induced antisymplectic involution $f\in \Bir(\calJ_V)$.  It follows that $\tau\circ f$ is a nontrivial symplectic birational involution of $\calJ_V$. Set $\widetilde{f}:=f$ if $f$ is symplectic, $\widetilde{f}:=\tau\circ f$ otherwise. Note that $\widetilde{f}$ leaves $\langle \Theta, \pi^*\calO(1)\rangle$ invariant in both cases. 
\end{proof}

\begin{remark}
    Fix the notation of the last paragraph of the proof of Theorem \ref{reduce to cubic}. If $\phi$ and thus $f$ is antisymplectic, then the symplectic birational involution $\widetilde{f}$ acts by $-\textnormal{id}$ on the discriminant group of $\OG$.
\end{remark}

\section{Nontrivial action on discriminant, rank \texorpdfstring{$\geq 12$}{>=12}}\label{sec: non-trivial > 12}

As noted in Remark \ref{remaining}, it remains to classify symplectic birational involutions of manifold of $OG10$ type that act nontrivially on $A_\Lambda$, with $\rank(\OG_-)\geq 12$ and the additional assumption that $\OG_+$ does not split a $U$ summand. 
The main aim of this section is to prove the following result:
\begin{theorem}\label{non-trivial non split}
	Let $X$ be a manifold of $OG10$ type, and $f\in \Bir(X)$ a symplectic birational involution such that the induced action $\eta_*(f)$ is nontrivial on the discriminant group $A_\OG$. Assume further that $$\rank\OG_-\geq 12.$$ Then one of the following holds:
    \begin{enumerate}
        \item $\OG_-\cong G_{12}$ and $\OG_+\cong \langle 2\rangle^3\oplus\langle -2\rangle^ 9$; or
        \item $\OG_-\cong G_{16}$ and $\OG_+\cong \langle 2\rangle^3\oplus\langle -2\rangle^ 5$.
    \end{enumerate}
 Here, $G_{12}$ and $G_{16}$ are the lattices listed in the last column of Table \ref{table enum}.

 Moreover, each involution of $\Lambda$ listed above is unique up to conjugacy in $O(\OG)$, and there exists a manifold of $OG10$ type with a birational involution inducing such an isometry.
\end{theorem}

We prove Theorem \ref{non-trivial non split} as follows: in Section \ref{sec: genus} we classify the possible genera of the coinvariant lattice $\OG_-$ for an involution $\iota\in O(\Lambda)$ with nontrivial action on $A_\OG$, and such that $\OG_-$ has rank at least 12, is negative definite and the lattice $\OG_+$ does not split a $U$ summand.
We obtain 12 possible genera for the lattice $\OG_-$. 
For each genera, we enumerate representatives for all the isometry classes of lattices in this genus. 
We outline the process in Section \ref{sec: enumeration} (more details are contained in Appendix \ref{appendix:A}). 
We then investigate whether each candidate coinvariant lattice occurs for a symplectic birational involution of a manifold of $OG10$ type. 
By Theorem \ref{thm: no markman so sympl} and Lemma \ref{lem: characterisation long roots}, this is equivalent to verifying whether the potential coinvariant lattice $\OG_-$ contains no vectors $v$ such that $v^2=-2$, or $v^2=-6$ with $\textnormal{div}_{\OG_-}(v)\in 3\bZ$.

Our results of this enumeration and analysis are summarised in Table \ref{table enum} in Section \ref{sec: results}; for more explicit details see the database \cite{marquand_lisa_2023_7528193}. 

\subsection{Genus of the remaining possible cases}\label{sec: genus}
We classify the possible genera of the coinvariant lattice $\OG_-$ for an involution $\iota\in O(\OG)$ with nontrivial action on $A_{\OG}$, such that $\OG_-$ is negative definite, and such that  $\OG_+$ does not split a $U$ summand.

\begin{proposition}\label{no U summand}
	Let $\iota\in O(\OG)$ acting nontrivially on $A_{\OG}$ such that $\OG_{-}$ is negative definite. Assume that $\OG_+$ does not split a $U$ summand. Let $r:=\rank\OG_{-}.$ Then one of the following holds:
	\begin{enumerate}
		\item $\OG_+\cong U(2)^3$ and $r=18$;
		\item $\OG_+\cong U(2)^3\oplus D_4$ and $r=14$; or
		\item $\OG_+\cong\langle 2\rangle^3\oplus \langle -2\rangle^{21-r}$ and $r\geq 12$. 
	\end{enumerate}
\end{proposition}
\begin{proof}
	For ease of notation, let $M:=\OG_+$.
	Since $M$ does not split a $U$ summand, the first part of the proof of Theorem \ref{partial class} gives that $r\geq 12$. Assume first that $r\neq 21$: by Proposition \ref{properties of Lneg}, $M$ is a 2-elementary lattice of signature $(3, 21-r)$. It is therefore uniquely determined, up to isometry, by the invariants $(r, a,\delta)$ (Lemma \ref{lem: 2elem}).  There are two cases to consider by Lemma \ref{split U}: either $22-r<a$, or $a=22-r$,  $\delta=0$ and $r\not\equiv 2\mod 8$. 
	
	\textbf{Case 1:} Assume that $22-r<a$; we necessarily have that $22-r<a\leq 24-r$. Since $M$ is 2-elementary, we have that $a\equiv r\mod 2$; we can exclude the case $a=23-r$. Thus $a=24-r=\rank M$. The lattice $N:=M(1/2)$ is well defined \cite[Lemma A.7]{marquand2022cubic}. Further, $A_N=\{1\}$, and so $N$ is unimodular. 
	
	Assume that $\delta=0$; this implies that $N$ is an even unimodular lattice (see for example \cite[Lemma A.9]{marquand2022cubic}). By Milnor's theorem on unimodular forms (see \cite[Thm 0.2.1]{nikulin} for a precise statement), $N$ exists if and only if 
	\begin{align*}
		3+r-21&\equiv 0\mod 8;\\
		r&\equiv 2\mod 8.
	\end{align*} 
	Since $r\geq 12$, we have that $r=18$. Thus $N$ has signature $(3,3)$, and hence $N\cong U^3$. Thus $M\cong U(2)^3.$
	
	Now assume that $\delta=1$. It follows that $N$ is an odd indefinite unimodular lattice (see for example \cite[Lemma A.8]{marquand2022cubic}). By Milnor's theorem again, $N$ exists and is isomorphic to $\langle 1\rangle^3\oplus\langle -1\rangle^{21-r}$, thus 
	$$M\cong \langle 2\rangle ^3\oplus \langle -2\rangle^{21-r}.$$
	
	\textbf{Case 2:} Assume that $a=22-r$, with $\delta=0$ and $r\neq 2 \mod 8$.
	Note again that $r\geq 12$; if $r\leq 11$, since $22-r=a\leq r\leq 22-r$, we must have that $r=11$. But since $\delta=0$, for $M$ to exist $r\equiv 2 \mod 4,$ a contradiction.
	
	So $r\geq 12$, and since $r\equiv 2\mod 4$, $r\in \{14, 18\}$. By assumption, $r\neq 2\mod 8$, thus $r=14.$
	Hence $M$ has signature $(3,7)$ with $a=8$ and $\delta=0$. Consider the lattice $U(2)^3\oplus D_4$; it has the same signature and invariants. Since indefinite 2-elementary lattices are unique up to isometry, we necessarily have 
	$M\cong U(2)^3\oplus D_4.$
	
	Finally, assume that $r=21$. In this case, $M$ has signature $(3,0)$. Since $a\equiv r\mod 2$, we have $a=1,\,3$. By Lemma \ref{lem: 2elem}, Item (3), the case $a=1$ cannot occur; thus $a=3$. Again, the lattice $N:=M(1/2)$ is well defined. Further, $A_N=\{1\},$ so $N$ is unimodular. Since there are no even unimodular lattices of rank 3, by Milnor's result, we must have that $\delta=1$ and $N$ is an odd unimodular lattice. Thus $N\cong \langle 1\rangle^3$, and 
	$M\cong \langle 2\rangle^3.$
\end{proof}

We now list the possible genera of $\OG_{-}$ for the involutions above.
\begin{corollary}\label{cases to enumerate}
	Let $\iota\in O(\OG)$ acting nontrivially on $A_{\OG}$ such that $\OG_{-}$ is negative definite. Assume that $\OG_+$ does not split a $U$ summand. Let $r:=\rank\OG_{-}.$ Then the discriminant group is $A_{\OG_-}=\bZ/3\bZ\oplus(\bZ/2\bZ)^a$, and $q_{\OG_-}|_{\bZ/3\bZ}= q_{A_2(-1)}$. Let $\delta=0$ if the quadratic form of $\OG_+$ takes values in $\bZ/2\bZ$, $\delta=1$ otherwise. Further, the invariants $(r, a, \delta)$ are as follows:
	\begin{enumerate}
		\item $(r, a , \delta) = (18, 6, 0)$
		\item $(r, a, \delta)=(14, 8, 0)$
		\item $(r, a, \delta)= (r, 24-r, 1)$ for $12\leq r\leq 21$.
	\end{enumerate}
\end{corollary}
\begin{proof}
	This follows immediately from Proposition \ref{no U summand}.
\end{proof}

 In order to conclude our classification of symplectic birational involutions for manifolds of $OG10$ type, it remains to be seen whether an involution $\iota\in O(\OG)$ as in Proposition \ref{no U summand} is induced by a symplectic birational involution. 
 By Theorem \ref{thm: no markman so sympl} and Lemma \ref{lem: characterisation long roots}, we need to determine whether $\OG_{-}$ contains any vector $v$ such that $v^2=-2$, or $v^2=-6$ with $\textnormal{div}_{\OG_-}(v)\in 3\bZ$. 
 One possible strategy is to classify the possibilities for the lattices $\OG_{-}$. Unfortunately, these lattices are not unique in their respective genus, and have both large rank and discriminant (the methods of Conway--Sloane have not been extended \cite{MR965484}). We undertake this enumeration in the next subsection \ref{sec: enumeration}, but first we illustrate this difficulty with two examples.

\begin{example}
	Consider $\OG_+\cong U(2)^3$. Then $\OG_{-}$ has rank $18$, and discriminant group $$A_{\OG_{-}}\cong \bZ/3\bZ\times (\bZ/2\bZ)^6,$$ and $q_{\OG_{-}}|_{\bZ/3\bZ}=q_{A_2(-1)}.$ There are two easily identifiable possibilities for $\OG_{-}$: 
	\begin{align*}
		A_2(-1)&\oplus K;\\
		A_2(-1)&\oplus E_8(-1)\oplus N,
	\end{align*} where $K$ is the Kummer lattice and $N$ is the Nikulin lattice (see \cite{MR728142} for a description of these lattices). Both of these embed into the lattice $\OG$ and are orthogonal to $U(2)^3$. Although both examples contain vectors of square $-2$ and thus cannot be realised by a geometric birational involution, there may be other lattices in the same genus without vectors of square $-2$, or of square $-6$ and divisibility $3$.
\end{example}

\begin{example}
	Consider $\OG_+\cong U(2)^3\oplus D_4(-1)$. Then $\OG_-$ has rank 14, and discriminant group
	$$A_{\OG_-}\cong \bZ/3\bZ\times (\bZ/2\bZ)^8.$$
	Thus $\OG_-$ is in the same genera as the lattice $A_2(-1)\oplus N\oplus D_4(-1)$. Again this example contains $(-2)$-vectors.
\end{example}

\subsection{The enumeration}\label{sec: enumeration}

By Corollary \ref{cases to enumerate}, we have 12 possible genera to enumerate. This process is computer aided; we use a modified version of Kneser's neighbour method, first described in \cite{Kneser1957KlassenzahlenDQ}. Details about the application of this method for quadratic lattices over totally real number fields and its algorithmic implementation can be found in \cite[\S 2]{scharlau}. The main idea is that for each genus $g$ of integral definite lattices and for some suitable prime number $p$, there exists a so-called \emph{$p$-neighbour graph} \cite[Page 742]{scharlau}, which we denote here by $\text{Kne}_p(g)$, whose nodes represent the isometry classes of lattices in $g$. Two nodes of $\text{Kne}_p(g)$ are connected by an edge if for respective representatives $L,\, L'$ of the corresponding isometry classes, $L$ and $L'$ are \emph{$p$-neighbours} \cite[(28.2)]{KneserMartin2002QF}. The enumeration of the isometry classes in $g$ can be done by walking through $\text{Kne}_p(g)$ and iteratively computing representatives for all the nodes. We refer to Appendix \ref{appendix:A} for more details on Kneser's algorithm and its modified implementation. 

We represent each isometry class of lattices obtained by the previous enumeration by their Gram matrix, which are available in the files \cite{marquand_lisa_2023_7528193}. Each of them represent a coinvariant lattice $\OG_-$ for an involution $\iota\in O(\OG)$. It remains to verify whether they are induced from a geometric involution or not.

 Let $L$ be one of the lattices enumerated; $L$ is the coinvariant lattice for an involution $\iota\in O(\OG)$. We verify whether $L$ is induced by a geometric involution by verifying if $L$ contains any vector $v$ such that $v^2=-2$, or $v^2=-6$ with $\textnormal{div}_{L}(v)\in 3\bZ$. 

To check whether $L$ contains vectors of square $-2$ or $-6$, one can use the method \textbf{short\_vectors} on Hecke \cite{nemo}, which  allows us to compute all vectors in $L$ of a given square. For each such vector $v\in L$ such that $v^2=-6$, computing the positive generator $d$ of the $\mathbb{Z}$-ideal $(v, L)$ is a simple routine which goes back to standard linear algebra (see also \cite[\textbf{divisibility}]{nemo}). 

Finally, if $L$ contains no vector $v$ such that $v^2=-2$, or $v^2=-6$ with $\textnormal{div}_{L}(v)=3$, then by Theorem \ref{thm: no markman so sympl} and Lemma \ref{lem: characterisation long roots}, we know that the involution $\iota$ is induced, and $L$ is isometric to the coinvariant lattice associated to a symplectic birational involution on an irreducible symplectic manifold of $OG10$ type.

\subsection{Proof of Theorem \ref{non-trivial non split}}\label{sec: results}
We give in Table \ref{table enum} the results of our enumeration and the analysis outlined above. For each possible genus, we give the number $N$ of isometry classes of lattices it contains. In the column \textsc{with roots} we record the number of classes that  have a representative with $(-2)$-vectors. In the column \textsc{without roots, \textbf{\textsc{but}} with $(-6,3)$} we record how many classes have a representative $L$ without any $(-2)$-vectors but with at least one vector $v\in L$ such that $v^2=-6$ and $\textnormal{div}_L(v) \in 3\bZ$. Finally, the last column \textsc{geometric cases} presents all possible isometry classes that are induced by a symplectic birational involution of a manifold $X$ of $OG10$ type, and thus are isometric to $H^2(X, \mathbb{Z})_-$. 

\begin{proof}[Proof of Theorem \ref{non-trivial non split}]
    Let $f\in \Bir(X)$ satisfy the assumptions of Theorem \ref{non-trivial non split} and let $\iota:=\eta_\ast(f)\in O(\OG)$. Then, according to Proposition \ref{no U summand}, Corollary \ref{cases to enumerate} and Table \ref{table enum}, we know that the pair $(\Lambda_+, \Lambda_-)$ is as wanted. Similarly, by \cite[Proposition 1.15.1]{nikulin}, and the verification from Section \ref{sec: enumeration}, we know that any involution of $\OG$ with $(\Lambda_+, \Lambda_-)$ as in the statement of Theorem \ref{non-trivial non split} is induced.

    To conclude, similarly to the proof of Theorem \ref{thm: prim embedding}, we observe that according to \cite[Theorem 1.14.2]{nikulin}, for $\OG_+$ as in the statement of Theorem \ref{non-trivial non split} we have that $O(\OG_+)\to O(A_{\OG_+})$ is surjective. Therefore, \cite[Corollary 1.5.2]{nikulin} tells us that any involution of $\OG$ with $(\Lambda_+, \Lambda_-)$ as in the statement of Theorem \ref{non-trivial non split} is unique up to conjugacy.
\end{proof}

\definecolor{lg}{gray}{0.9}
\newcolumntype{g}{>{\columncolor{lg}}c}
\renewcommand\arraystretch{1.3}
\begin{table}[!t]
	\centerline{
		\begin{tabular}{|c|g||g|g|g|g|}
			
			\hline
			\rowcolor{white}
			&&&&\textsc{without roots,}&\\
			\rowcolor{white}
			\multirow{-2}{*}{\textsc{Case}}&\multirow{-2}{*}{\textsc{genus}}&\multirow{-2}{*}{$N$}&\multirow{-2}{*}{\textsc{with roots}}&\textsc{\textbf{\textsc{but}} with $(-6,3)$}&\multirow{-2}{*}{\textsc{geometric cases}}\\
			
			\hline
			\rowcolor{white}
			(1)&$\textnormal{II}_{(0,18)}2^{+6}_{\textnormal{II}}3^{+1}$&430&430&0&None\\
			
			\hline
			(2)&$\textnormal{II}_{(0,14)}2^{-8}_{\textnormal{II}}3^{+1}$&21&21&0&None\\
			
			\hline
			\rowcolor{white}
			&$\textnormal{II}_{(0,12)}2^{-12}_23^{+1}$&5&4&0&1: $G_{12}$\\
			
			\hhline{~|-|-|-|-|-}
			&$\textnormal{II}_{(0,13)}2^{-11}_13^{+1}$&23&22&1&None\\
			
			\hhline{~|-|-|-|-|-}
			\rowcolor{white}
			&$\textnormal{II}_{(0,14)}2^{-10}_03^{+1}$&70&70&0&None\\
			
			\hhline{~|-|-|-|-|-}
			&$\textnormal{II}_{(0,15)}2^{-9}_73^{+1}$&211&211&0&None\\
			
			\hhline{~|-|-|-|-|-}
			\rowcolor{white}
			&$\textnormal{II}_{(0,16)}2^{-8}_63^{+1}$&617&616&0&1: $G_{16}$\\
			
			\hhline{~|-|-|-|-|-}
			&$\textnormal{II}_{(0,17)}2^{-7}_53^{+1}$ &1291&1291&0&None\\
			
			\hhline{~|-|-|-|-|-}
			\rowcolor{white}
			&$\textnormal{II}_{(0,18)}2^{-6}_43^{+1}$&2524&2524&0&None\\
			
			\hhline{~|-|-|-|-|-}
			& $\textnormal{II}_{(0,19)}2^{-5}_33^{+1}$&3682&3682&0&None\\
			
			\hhline{~|-|-|-|-|-}
			\rowcolor{white}
			&$\textnormal{II}_{(0,20)}2^{-4}_23^{+1}$&3375&3375&0&None\\
			
			\hhline{~|-|-|-|-|-}
			\multirow{-10}{*}{(3)}&$\textnormal{II}_{(0,21)}2^{-3}_13^{+1}$&1316&1316&0&None\\
			
			\hline
		\end{tabular}
	}
	\caption{Genus enumeration and geometric cases of Theorem \ref{non-trivial non split} }\label{table enum}
\end{table}
The two lattices admitting a geometric realization from Table \ref{table enum} are determined respectively by the following Gram matrices: 

\[\resizebox{\textwidth}{!}{
	$G_{12}= \begin{bmatrix}
		-4 &2 &-2 &-2 &-2& -2 &2& -2 &2& 2 &-2 &-2\\
		2 &-4 &2 & 0 & 2&0 & -2 &2& -2& -2 &2 &2\\
		-2 &2 &-4 &-2 &-2& -2& 2& -2 &2& 2 &-2 &-2\\
		-2 &0&-2 &-4& -2& -2& 0& 0& 0& 0& 0& 0\\
		-2 &2 &-2& -2& -4& -2 &2& -2& 2& 2 &-2& -2\\
		-2 &0& -2 &-2& -2& -4& 2& -2& 2& 2 &-2 &-2\\
		2& -2& 2 &0& 2& 2& -4& 2& -2& -2 &2 &2\\
		-2 &2& -2& 0& -2& -2 &2 &-6& 4 &4& -2& -4\\
		2 &-2& 2 &0& 2 &2& -2& 4 &-6& -2 &4& 4\\
		2& -2 &2& 0& 2 &2 &-2& 4& -2& -6& 2& 2\\
		-2& 2 &-2 &0&-2 &-2 &2 &-2 &4 &2& -6& -2\\
		-2 &2& -2 &0& -2& -2 &2& -4& 4& 2 &-2& -6
	\end{bmatrix}$,\; 
	$G_{16} =\begin{bmatrix}
		-4 &2 &-2 &2 &-1 &1 &2 &1 &2 &1 &-2 &1 &-2 &-2 &1 &-2\\
		2 &-4 &0 &-1 &2 &-2 &0 &-2& 0 &-2 &0 &-2 &1 &2 &1 &1\\
		-2 &0 &-4 &2 &1 &1 &1 &-1 &2 &-1 &-1& 1 &-1 &-1& 2& -2\\
		2 &-1 &2 &-4 &-1 &-2 &-2 &-1 &0 &1 &1 &0 &2 &0 &-2& 1\\
		-1 &2 &1 &-1 &-4 &1 &-1 &2& 1 &2 &1 &2 &1 &-1 &-2 &0\\
		1 &-2 &1 &-2 &1& -4 &-1 &-2 &0& -1 &-1& -2 &1& 1& -1 &1\\
		2 &0 &1 &-2 &-1 &-1 &-4&-1 &0 &0 &2 &1 &2 &1& -2 &1\\
		1 &-2 &-1 &-1 &2 &-2& -1 &-4& 0& -2 &-1& -1 &1& 0 &1 &0\\
		2 &0 &2 &0 &1 &0 &0 &0 &-4 &-1 &0 &-1 &1& 1& -1 &2\\
		1 &-2 &-1 &1& 2& -1 &0 &-2& -1& -4& 0& -2& 0 &2 &1 &0\\
		-2 &0 &-1 &1 &1& -1 &2& -1& 0& 0& -4 &-1 &-1 &-1& 1 &0\\
		1 &-2& 1 &0& 2& -2& 1 &-1 &-1 &-2& -1& -4& -1& 2 &1& 1\\
		-2 &1 &-1& 2 &1 &1 &2 &1 &1 &0& -1 &-1 &-4 &0 &2& -2\\
		-2 &2 &-1 &0& -1 &1 &1& 0& 1 &2 &-1 &2 &0 &-4 &0 &-1\\
		1 &1& 2& -2 &-2& -1 &-2& 1& -1 &1& 1 &1 &2 &0 &-4 &1\\
		-2 &1 &-2 &1 &0 &1 &1 &0 &2 &0 &0 &1 &-2 &-1& 1 &-4 
	\end{bmatrix}$}.
\]

They correspond respectively to the 3rd and the 472nd lattices of the respective cases \textbf{c3r12} and \textbf{c3r16} of our database, available in \cite{marquand_lisa_2023_7528193}. We moreover display in Table \ref{case3rank12} the Gram matrices for 5 representatives of the isometry classes in the genus of $G_{12}$ (only the nonzero entries are shown).

\appendix
\section{Enumeration of definite genera for \texorpdfstring{$\mathbb{Z}$}{Z}-lattices}\label{appendix:A}

The \emph{genus} of a lattice is the collection of all full rank integral lattices in a given regular quadratic space which are pairwise locally isometric at each place of $\mathbb{Q}$. We usually refer to a lattice by its isometry class, and we represent a genus by the isometry classes of the lattices contained in it. By \cite[Satz (21.3)]{KneserMartin2002QF}, a genus consists of finitely many isometry classes of lattices, though the genus itself is infinite. In order to continue the classification of symplectic birational involutions of manifolds of OG10 type, via Theorem \ref{thm: no markman so sympl}, we want to enumerate the following 12 genera given by Corollary \ref{cases to enumerate} (the notation follows Conway--Sloane convention for genus symbols, see \cite[Chapter 15]{splg}):
\begin{enumerate}
	\item $g = \text{II}_{(0, 18)}2^{+6}_{\text{II}}3^{+1}$;
	\item $g = \text{II}_{(0,14)}2^{-8}_{\text{II}}3^{+1}$;
	\item $g = \text{II}_{(0, r)}2^{-(24-r)}_{\delta}3^{+1}$ for $12\leq r\leq 21$ and $\delta\equiv 6-r\mod 8$.
\end{enumerate}
These are nonempty genera of even negative definite lattices of rank $\geq 12$. Let $g$ be one of the previous genera. In what follows, we explain how one can algorithmically enumerate all isometry classes in $g$.

\subsection{Kneser's Neighbour Algorithm}
Throughout, let $L$ be a lattice in the given genus $g$, let $V$ be the ambient quadratic space associated to $g$ ($L\subseteq V$) and let $p$ be a prime number not dividing $\text{det}(L)$. 
\begin{definition}[{{{\cite[\S 28]{KneserMartin2002QF},\cite[\S 5]{kir16}}}}]\label{defA1}
	Let $L'$ be a full rank lattice in $V$. Then $L'$ is called a \emph{$p$-neighbour} of $L$ if 
	$$L/(L\cap L') \cong L'/(L\cap L')\cong \mathbb{F}_p$$
	as $\mathbb{F}_p$-vector spaces.
\end{definition}

Let $b$ the underlying nondegenerate symmetric bilinear form of $L$. We call an element $y\in L$ \textbf{$p$-admissible} if  $y\in L\setminus pL$ and $b(y,y)\in 2p^2\mathbb{Z}$ . 

\begin{theorem}[{{{\cite[\S 1]{RSP}, \cite[\S 28]{KneserMartin2002QF}}}}]\label{theoapp}
	Let $y\in L$ be $p$-admissible and let
	\begin{equation}\label{neighbour}
		L(y) := L_y + \mathbb{Z}\frac{1}{p}y\;\;\text{where}\;\;L_y = \left\{ x\in L\; \mid \; b(x,y)\in p\mathbb{Z}\right\}.
	\end{equation}
	Then $L(y)$ is a $p$-neighbour of $L$. Moreover, $L(y) \in g$ and any $p$-neighbour of $L$ lying in $g$ is of the form $L(y)$ for some $p$-admissible vector $y\in L$.
\end{theorem}

We define $C(g)$ to be the set of all the (finitely many) isometry classes of lattices in $g$, $E_p(g)$ to be the set
\[E_p(g):=\left\{([L], [L'])\in C(g)^2\;\mid\; L,L'\;\text{are  $p$-neighbours}\right\}\]
and, finally, $\text{Kne}_p(g):= (C(g), E_p(g))$ to be the \textbf{$p$-neighbour graph} of $g$ (see \cite[page 742]{scharlau}). 

It follows that $\text{Kne}_p(g)$ consists of finitely many nodes given by $C(g)$, and two nodes $[L], [L']$ are connected by an edge if and only if $([L],[L'])\in E_p(g)$. The number of connected components of $\text{Kne}_p(g)$ does not depend on $p$, and each of them is an union of what are called \emph{proper spinor genera} of $g$ (see \cite[\S 102]{QuadForm} for the notion of spinor genera). One can show algorithmically (see \cite[Corollary 5.4.10]{kir16}) that the 12 genera we aim to enumerate consist of one proper spinor genus, meaning that their respective $p$-neighbour graphs are connected, for any prime number $p$. So in what follows, we assume that genera consists of only one proper spinor genus.

The neighbour construction (Eq.\hspace{-3pt} \ref{neighbour}) offers a pratical way to construct representatives of isometry classes in $g$ and it is the starting point for Kneser's neighbour algorithm (see \cite[Satz 28.4]{KneserMartin2002QF}). In particular, this algorithm lets us reconstruct $C(g)$ starting from a single isometry class $[L]$ and any prime number $p\nmid \text{det}(L)$.
\begin{remark}
	Constructing a first representative lattice $L$ of a given genus $g$ is feasible but nontrivial: we refer to \cite[\S 3.4, \S 3.5]{kir16} for more details.
\end{remark}
The idea is the following: we start by constructing all the $p$-neighbouring isometry classes $[L']$ of $[L]$ obtained from the neighbour construction. Then, we iterate this process to all the new isometry classes obtained until we have exhausted the $p$-neighbour graph, i.e. we cannot construct any new isometry class from the ones already obtained.  Note that an isometry class can be reached by several different other ones in $\text{Kne}_p(g)$: one has to compare any new neighbour lattice to representatives of the isometry classes already computed. The following result ensures that this process allows us to recover $C(g)$.
\begin{theorem}[{{{\cite[\S 1 (ix)]{RSP}}}}]
	If the genus $g$ consists only of one (proper) spinor genus, $\rank L \geq 3$ and $L\otimes \mathbb{Q}_p$ is isotropic, then the previous procedure returns representatives of all isometry classes in $g$.
\end{theorem}

However, classical implementations of this algorithm do not exhaust all the edges in the $p$-neighbour graph. In fact, for genera of definite lattices, there exists an invariant, called the \emph{mass}, which allows us to determine whether or not we have enumerated $C(g)$.
\begin{definition}
	Let $g$ be a genus of definite lattices. We call the \emph{mass} of $g$, which we denote $m(g)$, the following sum
	\begin{equation}
		m(g) := \sum_{[L]\in C(g)}\frac{1}{\#O(L)}.
	\end{equation}
	For any $[L]\in C(g)$, we call the term $w([L]):= \frac{1}{\#O(L)}$ the \emph{weight} of $[L]$ in $g$.
\end{definition}
Thanks to the Smith--Minkowski--Siegel formula (see \cite[\S 2, Eq. (2)]{MR965484}), one can actually compute the mass of a genus of definite lattices without enumerating its isometry classes. Therefore, while walking through $\text{Kne}_p(g)$, one can associate to each visited node the weight of the corresponding isometry class. Adding the weights of the isometry classes already found gives us information on which proportion of $C(g)$ we have enumerated. In particular, if the previous sum agrees with the mass of $g$, then all the isometry classes in $g$ have been enumerated.

\subsection{Practical Implementation}
We have explained the needed theory behind the enumeration of the genus $g$ (we refer to \cite{scharlau} for more details). Let us now make some comments on practical implementation and possible improvements. 

First of all, recall that $p$ is a prime number not dividing $\text{det}(L)$ and let $n$ be the rank of $L$. We have the following proposition:
\begin{proposition}[{{{\cite[Hilfssatz (28.7)]{KneserMartin2002QF}}}}]
	If $y, y'\in L$ are $p$-admissible and such that $[y] = [y']\in L/pL$, then $L(y) = L(y')$
\end{proposition}
Therefore, in practice, one only has to enumerate isotropic lines $l$ in $L/pL \cong \mathbb{F}_p^n$ and choose any $p$-admissible lift in $L$ of a representative of $l$ to construct the corresponding neighbour lattice. The set $\mathcal{L}_{p,n}$ of lines in $\mathbb{F}_p^n$ is of cardinality $(p^n-1)/(p-1)$. In the context of this paper, where $p\geq 5$ and $n\geq 12$, enumerating isotropic elements in $\mathcal{L}_{p,n}$ at each iteration of the neighbours construction is infeasible: for instance, in the case $p=5$ and $n=15$, $\mathcal{L}_{5,15}$ is of size 7629394531. The whole process of enumerating all the isotropic lines, looking for $p$-admissible lifts, constructing the corresponding neighbours, comparing each new neighbour to the already visited nodes in $C(g)$ and computing the weight of every new visited node can take, in this case, up to 5 hours. If we do this process multiple times for different isometry classes, depending on the size of $C(g)$, already for rank 15 the enumeration can take several days. Thus, one could use in complement the second part of \cite[Hilfssaftz (28.7)]{KneserMartin2002QF} stating that if $y,y'\in L$ are $p$-admissible and if there exists $\phi\in O(L)$ such that $\phi(y) = y'$, then $L(y)$ and $L(y')$ are isometric. In other words, walking through $\text{Kne}_p(g)$ only requires to lift representatives of $O(L)$-orbits of isotropic elements in $\mathcal{L}_{p,n}$ since two isotropic lines in a same orbit give rise to isometric neighbour lattices, which correspond to the same class in $C(g)$. However, still in the context of this paper, even though $\#O(L)$ is finite for a definite lattice $L$, it turns out that the available algorithms (to our knowledge) to compute $O(L)$-orbits in $\mathcal{L}_{p,n}$ require either too much memory space or too much computation time, because of the high ranks and determinants of the lattices to find.

\subsection{Further improvements}
A way to bypass the complexity matters enunciated in the previous section, is to add randomisation. We iteratively proceed as follows: we choose a random isometry class among those already visited in $C(g)$, a random prime number $p$ not dividing $\text{det}(L)$ and we collect a large sample $N >> 0$ of random elements in $\mathcal{L}_{p,n}$ to construct neighbour lattices in $\text{Kne}_p(g)$ from the isotropic lines among them. Selecting different prime numbers allow us to enumerate $C(g)$ by finding neighbour lattices in different neighbour graphs in parallel. This can help if some lattices are connected by fewer edges in $\text{Kne}_p(g)$ than in $\text{Kne}_q(g)$ for different primes $p$ and $q$. The enumeration of genera using Kneser's neighbour method has been implemented on Hecke \cite{nemo} and we adapt this code with the random search just mentioned.

A final issue one can face with this randomised algorithm is that random walks in $\text{Kne}_p(g)$ can eventually lead to vain iterations. In this case, there might be only a little number of nodes which have not been visited yet and the probability to find them via random construction of neighbour lattices is low. This is the case, for instance, for isometry classes of lattices with relatively small weight:  their isometry group has big cardinality and therefore, there are few edges connecting to the corresponding node in $\text{Kne}_p(g)$ for all small primes $p$. It is possible, at this stage, to complete the enumeration of $C(g)$ by hand. To do so, let $C_{al}(g)$ be the list of representatives of isometry classes in $g$ which have been already computed, and let 
\[ m := m(g) - \sum_{[L]\in C_{al}(g)}w([L])\in \mathbb{Q}.\]
Let $q$ be the biggest prime number dividing the denominator of $m$. This prime $q$ must divide the order of the isometry group of a class in $C(g)\setminus C_{al}(g)$, and thus such isometry class contains a lattice with at least one isometry of order $q$. In the paper \cite{bt21}, Brandhorst and Hofmann have developed methods to compute, given a genus $g$ and a prime number $q$, pairs $(L,f)$ consisting of a lattice $L\in g$ and an isometry $f\in O(L)$ of order $q$. Their algorithms have been implemented using a combination of Sage \cite{sagemath}, GAP \cite{GAP4}, Magma \cite{magma} and PARI \cite{PARI98} (see also \cite[QuadFormWithIsom]{OSCAR} for a more recent and centralized implementation by the second author). Running such codes, we can find new lattices in $g$ not represented by any element of $C_{al}(g)$, and having an isometry of order $q$. Subtracting the weights of their respective isometry classes to $m$, we should eventually be able to clear out $q$ from the divisors of the denominator of $m$. If $m$ is still nonzero, we keep going with the next biggest prime $q'$ dividing the denominator of $m$, and so on until $m=0$. If $m$ happens to be an integer after some iteration, we continue with the largest prime number smaller than the one previously tried.

\subsection{Results and verification}
Combining both the enumeration by random searches and the construction of lattices with an isometry of a given order, we enumerate each of the 12 genera mentioned at the beginning of this Appendix. The lattices we have enumerated are available in the folder ``enumeration" in \cite{marquand_lisa_2023_7528193}. In the latter, each genus $g$ is given as a file, and each line of this file corresponds to the Gram matrix of a representative of an isometry class in $g$. The different ``reading files" made available allow one to read those genera on Hecke \cite{nemo}, Magma \cite{magma}, Oscar \cite{OSCAR} and Sage \cite{sagemath}.\smallskip

In order to verify our results, we refer to the folder ``verification" in \cite{marquand_lisa_2023_7528193}. In this folder, we provide a notebook providing step-by-step Hecke instructions to verify:
\begin{itemize}
    \item that each genus is correctly enumerated (see below for a brief explanation);
    \item the number of vectors of square $-2$, or of square $-6$ and divisibility at least 3 in each lattice.
\end{itemize}
Note that in this folder, the genera are saved differently. Now, each genus $g$ is given by a folder, and each file in this folder correspond to a representative of an isometry class in $g$. In each such file, we remember the Gram matrix of the lattice $L$ as well as the order of its orthogonal group $O(L)$.

\begin{remark}\label{rem: order isometry}
    Generators for the orthogonal groups of definite lattices are effectively computable \cite{plesken-souvignier}. However, such computations can be expensive for large rank lattices. In our enumeration process, which was originally purely written on Hecke, we made use of Magma for computing orthogonal group and testing isometry of lattices. For one's convenience, we have precomputed the order of the orthogonal groups of every lattices in our database using Magma \cite{magma} and stored it.
\end{remark}
In order to check whether the genera in \cite{marquand_lisa_2023_7528193} are correctly enumerated, one can proceed in the following way. 

Given a list $S$ of lattices, and a genus $g$, one ensures that $S$ is a complete set of representatives of isometry classes in $g$ without repetition, by
\begin{enumerate}
    \item checking that for all $L\in S$, the genus of $L$ is $g$;
    \item if one denotes $m$ the mass of $g$, computing that
    \[\sum_{L\in S}\frac{1}{\#O(L)} = m;\]
    \item testing whether lattices in $S$ are pairwise non isometric.
\end{enumerate}
In order to test whether two given definite lattices in a same genus are isometric, one can use algorithms of Plesken--Souvignier \cite{plesken-souvignier}. This can be challenging for lattices of high rank. Nonetheless, there is actually a way of making this procedure less challenging, by working with isometry invariants.

In fact, for instance, two lattices whose orthogonal groups have different order cannot be isometric. Hence, up to choosing a certain list of isometry invariants which are not hard to compute, one can definitely speed up step (3) and only call Plesken--Souvignier algorithm for comparing lattices with similar invariants. In our case, we have used:
\begin{itemize}
    \item the order of the orthogonal group of the lattice;
    \item the minimum norm, in absolute value, of a vector in the lattice; 
    \item the \textbf{kissing number} $k$ of the lattices, which is equal to the number of vectors of minimum (absolute) norm;
    \item the decomposition of the \textbf{root sublattice} of the lattice (see \cite{ebeling});
    \item the numbers of vectors of norm $-4$, and $-6$ with divisibility in $3\bZ$.
\end{itemize}
More details in \cite{marquand_lisa_2023_7528193}.

\bibliography{OG10--new}

\begin{table}[ht!]
	\caption{Case 3, $r=12$}
	\label{case3rank12}
	\resizebox{!}{0.55\textheight}{\begin{tabular}{ c}
		\toprule
		\textbf{Lattices in the isometry class of Corollary \ref{cases to enumerate} with $(r,a,\delta)=(12, 12, 1)$}\\
		\midrule\\
		\addlinespace[-2ex]
		$L_{12,12}^1 = \begin{bmatrix} 
			-2 & & & && & &&& & &\\ 
			&-2 & && & & & & & & & \\
			&& -2& & & & & & & & & \\
			& & &-2& & & & & & &  &\\
			&& & & -2& & & & & & & \\
			&& & & &-2& & & & & & \\
			& & & & & &-4& -2& -2& -2& -2& -2\\
			&& & & & & -2& -4& -2& -2& & \\
			& && & & & -2& -2& -4& & & -2\\
			& & & & & & -2& -2& & -4&  &\\
			& &&  & & &-2& & &  &-4& -2\\
			&&& & & & -2& &-2& & -2& -4 
		\end{bmatrix}$ \\
		\addlinespace[1.5ex]
		$L_{12,12}^2= \begin{bmatrix} 
			-2 & & & & & & & & & & &\\ 
			&-2&  & & && & &&  && \\
			&&-4& 2& -2 &-2 &2& -2 &2 &2&2& 2\\
			&  &2 &-4 & && &&  & &-2& -2\\
			&  &-2 && -4 && 2 &-2 &2 &2 &2& 2\\
			&&-2 &&& -4 && &  &&  &\\
			&& 2 &&2 & &-4& 2 &-2& -2 &-2 &-2\\
			& & -2 && -2 & &2& -4& 2 &2&  &\\
			& & 2 & &2 & &-2 &2& -4& -2& -2 &-2\\
			& & 2& & 2 &&-2 &2 &-2 &-4& &-2\\
			& & 2& -2& 2&  &-2& & -2 & &-6& -4\\
			& & 2& -2 &2&  &-2&  &-2& -2 &-4& -6
			
		\end{bmatrix}$\\
		\addlinespace[1.5ex]
		$G_{12}= \begin{bmatrix} 
			-4 &2 &-2 &-2 &-2& -2 &2& -2 &2& 2 &-2 &-2\\
			2 &-4 &2 && 2& & -2 &2& -2& -2 &2 &2\\
			-2 &2 &-4 &-2 &-2& -2& 2& -2 &2& 2 &-2 &-2\\
			-2 &&-2 &-4& -2& -2& & & & & & &\\
			-2 &2 &-2& -2& -4& -2 &2& -2& 2& 2 &-2& -2\\
			-2 && -2 &-2& -2& -4& 2& -2& 2& 2 &-2 &-2;\\
			2& -2& 2 && 2& 2& -4& 2& -2& -2 &2 &2\\
			-2 &2& -2& & -2& -2 &2 &-6& 4 &4& -2& -4\\
			2 &-2& 2 && 2 &2& -2& 4 &-6& -2 &4& 4\\
			2& -2 &2& & 2 &2 &-2& 4& -2& -6& 2& 2\\
			-2& 2 &-2 &&-2 &-2 &2 &-2 &4 &2& -6& -2\\
			-2 &2& -2 && -2& -2 &2& -4& 4& 2 &-2& -6
		\end{bmatrix}$\\
		\addlinespace[1.5ex]
		$L_{12,12}^4= diag(-2, -2, -2, -2, -2, -2, -2, -2, -2, -2, -2, -6)$\\
		%	\begin{bmatrix} 
			%		-2 & & & & & & & & & && \\
			%		 & -2 & & & & & & & & && \\
			%		  & & -2& & & & & & & & & \\
			%		   & & & -2 &&& & & & & & \\
			%		    & & & & -2& & & &  & & & \\
			%		    & & & & & -2& & &  & & &\\
			%		    & & & & & & -2& & & & & \\
			%		     & & & & & & & -2& & & & \\
			%		     & & & & & & & & -2& & &\\
			%		     & & & & & & & & & -2& & \\
			%		      & & && & & & & & & -2 &\\
			%		      & & & & & & & & & & & -6
			%	\end{bmatrix}$\\
		\addlinespace[1.5ex]
		$L_{12,12}^5= \begin{bmatrix} 
			-2 &&&&&&&&&&&\\
			& -4 &-2& -2& 2 &-2& 2& -2& 2&&& \\
			&-2 &-4& -2 && -2 &2&& & & &  \\
			& -2& -2& -4& 2& -2& 2 &-2& 2& &&\\
			& 2& & 2& -4& 2& & 2& -2& & & \\
			& -2& -2& -2& 2& -4& 2& -2& 2& & & \\
			& 2& 2& 2& & 2& -4& 2& -2& & & \\
			& -2& & -2& 2& -2& 2& -4& 2& & & \\
			&2 && 2& -2& 2& -2& 2& -4& & & \\
			& & & & & & & & & -2& & \\
			& & & & & & & & & & -2& \\
			& & & & & & & & & & & -6
		\end{bmatrix}$\\
		\bottomrule
	\end{tabular}}
\end{table}

 \end{document}